\documentclass[12pt,reqno]{amsart}

\usepackage[paper=letterpaper,margin=1in,twoside=false,includehead,headheight=18pt]{geometry}

\usepackage[T3,T1]{fontenc}

\usepackage{amsmath,amssymb,amsthm} 
\usepackage{enumerate,paralist}
\usepackage{graphicx}
\usepackage{mathtools}

\usepackage{todonotes}

\usetikzlibrary{decorations.pathreplacing}

\newtheorem{theorem}{Theorem}
\newtheorem{lemma}[theorem]{Lemma} 
\newtheorem{proposition}[theorem]{Proposition} 
\newtheorem{corollary}[theorem]{Corollary}

{\theoremstyle{definition} 

\newtheorem*{definition}{Definition} 
 
\newtheorem{example}[theorem]{Example} 
 }

\newtheoremstyle{named}%
  {}{}						
  {\upshape}				
  {0pt}{\bfseries}			
  {.}						
  {.5em}					
  {\thmname{#1}\thmnote{ #3}}  

\theoremstyle{named}

\DeclareSymbolFont{tipa}{T3}{cmr}{m}{n}
\DeclareMathAccent{\invbreve}{\mathalpha}{tipa}{16}

\newcommand{\wo}{\setminus} 
 
\newcommand{\down}[1]{\left\lfloor{#1}\right\rfloor} 

\newcommand{\card}{\#}
\newcommand{\set}[1]{\left\{{#1}\right\}} 
\newcommand{\setof}[2]{\left\{{#1}\,:\,{#2}\right\}}
\newcommand{\of}{\subseteq}

\newcommand{\union}{\bigcup}

\renewcommand{\P}{\mathbb{P}}

\newcommand{\A}{\mathcal{A}}
\renewcommand{\L}{\mathcal{L}}
\newcommand{\Hy}{\mathcal{H}}
\newcommand{\N}{\mathcal{N}}
\newcommand{\Nc}{\mathcal{N}_c}
\newcommand{\Ra}{\mathcal{R}}
\renewcommand{\Pr}{\mathcal{P}}

\newcommand{\hind}{H_{\text{ind}}}

\newcommand{\Eq}[1][q]{E^{\circ}_{#1}}

\newcommand{\D}{\mathcal{D}}

\renewcommand{\l}{\ell}
\newcommand{\lc}{\ell_c}

\newcommand{\la}{\lambda}

\newcommand{\chr}{\invbreve{\chi}}

\newcommand{\wl}{w^{\la}}
\newcommand{\Nl}{N^{\la}}
\renewcommand{\ll}{\l^{\la}}
\newcommand{\llc}{\lc^{\la}}

\DeclareMathOperator{\XHom}{XHom}
\DeclareMathOperator{\xhom}{xhom}
\DeclareMathOperator{\Hom}{Hom}

\DeclareMathOperator{\range}{range}
\DeclareMathOperator{\ds}{ds}
\DeclareMathOperator{\sds}{sds}
\DeclareMathOperator{\id}{id}
\DeclarePairedDelimiter{\abs}{\lvert}{\rvert}

\begin{document}

\title{Counting dominating sets and related structures in graphs}
\author{Jonathan Cutler}
\address[Jonathan Cutler]{Department of Mathematical Sciences, Montclair State University, Montclair, NJ 07043 USA}
\email{jonathan.cutler@montclair.edu}
\author{A.~J.~Radcliffe}
\address[A.~J.~Radcliffe]{Department of Mathematics, University of Nebraska-Lincoln, Lincoln, NE 68588 USA}
\email{jamie.radcliffe@unl.edu}
\maketitle
\begin{abstract}
	We consider some problems concerning the maximum number of (strong) dominating sets in a regular graph, and their weighted analogues.  Our primary tool is Shearer's entropy lemma.  These techniques extend to a reasonably broad class of graph parameters enumerating vertex colorings satisfying conditions on the multiset of colors appearing in (closed) neighborhoods.  We also generalize further to enumeration problems for what we call existence homomorphisms.  Here our results are substantially less complete, though we do solve some natural problems.
\end{abstract}

\maketitle

\section{Introduction} 
\label{sec:intro}

Many interesting problems arise when one asks which graphs maximize some graph invariant over a fixed class of graphs.  The history of these problems goes back to at least Mantel's theorem, where the class of graphs is that of triangle-free graphs on $n$ vertices and the graph invariant is simply the number of edges.  Recently, there has been a lot of interest in maximizing the number of independent sets in a graph from some class.  One of the earliest such results is due to Moon and Moser \cite{MM}.
\begin{theorem}[Moon, Moser]
	If $G$ is a graph on $n\geq 2$ vertices, then the number of maximal independent sets in $G$ is at most
	\[
		\begin{cases}
			3^{n/3} & \text{if $n\equiv 0 \pmod 3$},\\
			4\cdot 3^{\down{n/3}-1} & \text{if $n\equiv 1 \pmod 3$},\\
			2\cdot 3^{\down{n/3}} & \text{if $n\equiv 2 \pmod 3$}.
		\end{cases}
	\]
\end{theorem}
Another example of such a result is the following beautiful theorem due to Kahn \cite{K} and Zhao \cite{Z} concerning $i(G)$, the number of independent sets in a graph $G$.

\begin{theorem}[Kahn, Zhao]\label{thm:KZ}
	If $G$ is an $r$-regular graph on $n$ vertices, then 
	\[
		i(G)\leq i(K_{r,r})^{n/(2r)}=(2^{r+1}-1)^{n/(2r)}.
	\]
\end{theorem}

Kahn proved Theorem~\ref{thm:KZ} for bipartite graphs using entropy methods and Zhao extended this result to all regular graphs by a clever use of the bipartite double cover.  Motivated by a conjecture of Kahn about a weighted version of the bipartite case of Theorem~\ref{thm:KZ}, Galvin and Tetali \cite{GT} made the observation that independent sets are generalized by homomorphisms into a fixed image graph.  Recall that a \emph{homomorphism from $G$ to $H$} is a function $\phi:V(G)\to V(H)$ such that $xy\in E(G)$ implies $\phi(x)\phi(y)\in E(H)$.  Let $\Hom(G,H)$ be the set of homomorphisms from $G$ to $H$ and $\hom(G,H)=\abs{\Hom(G,H)}$.  If we let $\hind$ be the graph on two vertices with an edge between them and one of the vertices looped (see Figure~\ref{fig:hind}),
\begin{figure}[ht]
	\includegraphics{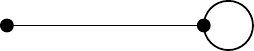}
	\caption{The graph $\hind$.}\label{fig:hind}
\end{figure}
then independent sets in $G$ correspond precisely to elements of $\Hom(G,\hind)$.  (An independent set $I$ corresponds to a homomorphism to $\hind$ in which the preimage of the unlooped vertex is $I$.)  Galvin and Tetali \cite{GT} proved the following.

\begin{theorem}[Galvin, Tetali]
	If $G$ is an $r$-regular bipartite graph on $n$ vertices and $H$ is any graph (which may have loops), then
	\[
		\hom(G,H)\leq \hom(K_{r,r},H)^{n/(2r)}.
	\]
\end{theorem}

A variety of homomorphism counting problems correspond directly to statistical physics models.  For instance, $\Hom(G,\hind)$ is the hard-core model.  In this context, it is common to weight homomorphisms by giving an activity to each vertex of the image graph.  To be precise, we have the following.

\begin{definition}
	Given a (possibly looped) graph $H$ and a function $\la:V(H)\to (0,\infty)$, we define the \emph{weight} of a homomorphism $\phi\in \Hom(G,H)$ to be
	\[
		w^{\la}(\phi)=\prod_{v\in V(G)} \la(\phi(v)).
	\]
	The analogue of the number of homomorphisms is the total weight of all homomorphisms from $G$ to $H$.  We write
	\[
		Z^{\la}(G,H)=\sum_{\phi\in \Hom(G,H)} w^{\la}(\phi).
	\]
\end{definition}

Galvin and Tetali \cite{GT} also showed the following.

\begin{theorem}[Galvin, Tetali]
	If $G$ is an $r$-regular bipartite graph on $n$ vertices, $H$ is any graph (which may have loops), and $\la:V(H)\to (0,\infty)$, then
	\[
		Z^{\la}(G,H)\leq (Z^{\la}(K_{r,r},H))^{n/(2r)}.
	\]
\end{theorem}

We are primarily concerned in this paper with bounds on the number of dominating sets and strong dominating sets in graphs, for instance in $r$-regular graphs on a given number of vertices or in graphs with a fixed number of vertices and edges.  Let us recall a few relevant definitions.  If $G$ is a graph, we let $N(v)$ and $N[v]$ be the open and closed neighborhoods of a vertex $v$, respectively.  Similarly, for $S\of V(G)$, we let $N(S)=\union_{v\in S} N(v)$ and $N[S]=\union_{v\in S} N[v]$.

\begin{definition}
	In a graph $G=(V,E)$, a set $S\of V$ is a \emph{dominating set} if $N[S]=V$.  We say $S$ is a \emph{strong dominating set} if $N(S)=V$.
\end{definition}

Fomin, Grandoni, Pyatkin, and Stepanov \cite{FGPS} were able to prove the following analogue of the Moon-Moser theorem, bounding the number of \emph{minimal} dominating sets in a graph.

\begin{theorem}[Fomin et al.]
	If $G$ is a graph on $n$ vertices, then the number of minimal dominating sets in $G$ is at most $1.7159^n$.
\end{theorem}

One of the tools that we will use is entropy.  Recall that if $X$ is a random variable, then the \emph{entropy of $X$} is the following property of the distribution of $X$:
\[
	H(X)=\sum_{x\in \range(X)} \P(X=x)\log \frac{1}{\P(X=x)},
\]
where this logarithm and all other logarithms in this paper are base two.  It is well-known that if $\abs{\range(X)}=n$, then $H(X)\leq \log n$ with equality if and only if $X$ is uniform on its range.  If $X=(X_1,X_2,\ldots,X_n)$ is a random sequence and $A\of [n]$, then we let $X_A=(X_i)_{i\in A}$ be the restriction of $X$ to $A$.  Our main entropy tool will be Shearer's Lemma \cite{CGFS}.

\begin{theorem}[Shearer's Lemma]
	If $X=(X_1,X_2,\ldots,X_n)$ is a random sequence and $\A$ is a multiset of subsets of $[n]$ such that every $i\in [n]$ is in at least $k$ elements of $\A$, then
	\[
		H(X)\leq \frac{1}k \sum_{A\in \A} H(X_A).
	\]
\end{theorem}

In this paper, we use Shearer's Lemma to prove upper bounds on the number of dominating sets and strong dominating sets in regular graphs.  We also show that these bounds extend to weighted dominating sets by generalizing the bounds to what we call legal closed neighborhood colorings.  We prove similar results about weighted strong dominating sets.  We derive some corollaries including results on the chromatic polynomial of the square of a regular graph and the number of rainbow colorings of neighborhood hypergraphs.  

We also generalize strong dominating sets to graph functions called \emph{existence homomorphisms}.  We prove bounds on the number of existence homomorphisms for some image graphs using elementary methods.  We conclude with some open questions.


\section{Dominating sets and strong dominating sets} 
\label{sec:domsets}

In this section, we give best possible bounds on the weighted number of dominating (and strong dominating) sets in an $r$-regular graph on $n$ vertices.  We do this by introducing the idea of a \emph{neighborhood legal coloring} of the vertices of a graph.  However, we will start by giving a quick proof of the unweighted version for dominating sets since it lays bare the essential techniques.  We let $\ds(G)$ be the number of dominating sets in $G$.

\begin{theorem}
	If $G$ is an $r$-regular graph on $n$ vertices, then
	\[
		\ds(G)\leq \ds(K_{r+1})^{n/(r+1)}.
	\]
\end{theorem}

\begin{proof}
	Let $S$ be a random set chosen uniformly from the dominating sets in $G$, and let $X=(X_v)_{v\in V(G)}$ be its characteristic vector.  We will apply Shearer's Lemma with the multiset $\A$ being the collection $\setof{N[v]}{v\in V(G)}$.  Note that each vertex in $G$ is in exactly $r+1$ sets in $\A$.  Also, since $S$ is a dominating set, it cannot be the case that $X_{N[v]}$ is identically $0$ for any $v\in V(G)$.  Thus, the random variable $X_{N[v]}$ takes at most $2^{r+1}-1$ values.  We have
	\begin{align*}
		\log \ds(G)&=H(X)\\
		&\leq \frac{1}{r+1} \sum_{v\in V(G)} H(X_{N[v]})\\
		&\leq \frac{1}{r+1} \sum_{v\in V(G)} \log(2^{r+1}-1)\\
		&= \frac{n}{r+1} \log(2^{r+1}-1)\\
		&= \frac{n}{r+1} \log(\ds(K_{r+1})),
	\end{align*}
	as we claim.
\end{proof}

The core idea that makes this proof work is that we can tell whether $S$ is a dominating set simply by examining $X_{N[v]}$ for each $v\in V(G)$.  We generalize this idea to vertex colorings of $G$ such that the restrictions to each (closed) neighborhood falls into some class of legal colorings.  The following definition makes this precise.

\begin{definition}
	Let $K$ be a finite set of colors and let $\L$ be a collection of multisets of $K$.  We say that $\phi:V(G)\to K$ is a \emph{$\L$-legal neighborhood coloring} of the graph $G$ if $\phi(N(v))\in \L$ for all $v\in V(G)$.  Note that we are considering the multiset image of $N(v)$ under $\phi$.  Similarly, we say that $\phi:V(G)\to K$ is a \emph{$\L$-legal closed neighborhood coloring} of the graph $G$ if $\phi(N[v])\in \L$ for all $v\in V(G)$.  We let $\l(G,\L)$ be the number of $\L$-legal neighborhood colorings of $G$ and $\lc(G,\L)$ the number of $\L$-legal closed neighborhood colorings of $G$.  We refer to $\L$ as a \emph{coloring condition}.
\end{definition}

\begin{example}
	If we set $K=\set{0,1}$ and let $\L$ be all multisets of $K$ containing at least one $1$, then an $\L$-legal neighborhood coloring of $G$ is precisely the characteristic function of a strong dominating set in $G$.  Similarly, an $\L$-legal closed neighborhood coloring of $G$ is a dominating set in $G$.
\end{example}

In all of our results in this section, we will be considering regular graphs, and therefore only one size of multiset in $\L$ will be relevant.  

It is straightforward to compute $\l(K_{r,r},\L)$ and $\lc(K_{r+1},\L)$ for any coloring condition $\L$, but we need introduce a bit of notation.  Let us write
\[
	N(r,\L)=\card\setof{f:[r]\to K}{f([r])\in \L},
\]
the number of functions from a domain of size $r$ whose multiset image is legal according to $\L$.  For $L\in \L$ of size $r$, we write $\binom{r}{n(L)}$ for the multinomial coefficient\footnote{Note that $n(L)$ is the multiset of repetition counts of elements of $L$.} counting the number of functions $f:[r]\to K$ such that $f([r])=L$.  For instance,
\[
	\binom{5}{n(\set{a,a,b,b,c})}=\binom{5}{2,2,1}=30.
\]
Thus,
\begin{equation*}
	N(r,\L)=\sum_{\substack{L\in \L\\\abs{L}=r}} \binom{r}{n(L)}.
\end{equation*}

An $\L$-legal neighborhood coloring of $K_{r,r}$ simply consists of a coloring that uses an element of $\L$ on each side.  Thus, $\l(K_{r,r},\L)=N(r,\L)^{2}$.  Similarly, an $\L$-legal closed neighborhood coloring of $K_{r+1}$ is one whose image is in $\L$, and so $\lc(K_{r+1},\L)=N(r+1,\L)$.

\begin{theorem}\label{thm:legalcount}
	If $G$ is an $r$-regular graph on $n$ vertices and $\L$ is a coloring condition, then
	\[
		\l(G,\L)\leq N(r,\L)^{n/r}=\l(K_{r,r},\L)^{n/(2r)},
	\]
	and
	\[
		\lc(G,\L)\leq N(r+1,\L)^{n/(r+1)}=\lc(K_{r+1},\L)^{n/(r+1)}.
	\]
\end{theorem}

\begin{proof}
	For the first inequality, let $\phi$ be a random $\L$-legal neighborhood coloring of $G$ chosen uniformly from the set of all such colorings, and let $\A$ be the multiset $\setof{N(v)}{v\in V(G)}$.  Applying Shearer's Lemma, we have
	\begin{align*}
		\log \l(G,\L)&=H(\phi)\\
		&\leq \frac{1}{r} \sum_{v\in V(G)} H(\phi|_{N(v)})\\
		&\leq \frac{1}{r} \sum_{v\in V(G)} \log N(r,\L)\\
		&= \frac{n}{r} \log N(r,\L)\\
		&=\frac{n}{2r} \log N(r,\L)^2\\
		&= \frac{n}{2r} \log(\l(K_{r,r},\L)),
	\end{align*}
	as we claim.
	
	In the proof of the second inequality, the covering we consider is $\A=\setof{N[v]}{v\in V(G)}$.  The rest of the proof is the same, \emph{mutatis mutandis}.
\end{proof}

It is standard in many enumeration problems of this type to consider a weighted version in which there are ``activations'' $\la_x$ associated to each $x\in K$.  We next make this framework precise and prove a weighted version of Theorem~\ref{thm:legalcount}.  

\begin{definition}
	Let $K$ be a finite set and $\la:K\to (0,\infty)$ be an \emph{activation function on $K$}.  We define weighted versions of $N(r,\L)$, $\l$, and $\lc$.  For a graph $G$ and $\phi:V(G)\to K$, let
	\begin{align*}
		\wl(\phi)&=\prod_{v\in V(G)} \la(\phi(v)),\\
		\ll(G,\L)&=\sum \wl(\phi),
		\shortintertext{where the sum is over all $\L$-legal neighborhood colorings $\phi$,}
		\llc(G,\L)&=\sum \wl(\phi),
		\shortintertext{this time summed over all $\L$-legal closed neighborhood colorings $\phi$, and}
		\Nl(r,\L)&=\sum_{\substack{\phi:[r]\to K\\ \phi([r])\in \L}} \wl(\phi)\\
		&=\sum_{\substack{L\in \L\\\abs{L}=r}} \binom{r}{n(L)}\prod_{x\in L} \la(x).
	\end{align*}
\end{definition}

\begin{theorem}\label{thm:legalwt}
	If $G$ is an $r$-regular graph on $n$ vertices, $\L$ is a coloring condition, and $\la:K\to (0,\infty)$ is an activation function, then
	\[
		\ll(G,\L)\leq \Nl(r,\L)^{n/r}=\ll(K_{r,r},\L)^{n/(2r)},
	\]
	and
	\[
		\llc(G,\L)\leq \Nl(r+1,\L)^{n/(r+1)}=\llc(K_{r+1},\L)^{n/(r+1)}.
	\]
\end{theorem}

\begin{proof}
	We prove the result for rational weights, the general case follows from continuity.  First observe that if we scale all weights by a positive factor $q$, then both sides of each inequality above scale by $q^{n}$.  Thus, we may clear denominators and assume all weights $\la(x)$ are integers.  We will introduce modified versions of $K$ and $\L$ as follows.  Let 
	\[
		K'=\setof{(x,i)}{x\in K, i\in [\la(x)]},
	\]
	and let $\L'$ be the set of multisets of $K'$ whose image under the projection onto the first coordinate is in $\L$.  We have
	\[
		\ll(G,\L)=\l(G,\L'),\qquad \llc(G,\L)=\lc(G,\L'),\qquad \text{and} \qquad \Nl(s,\L)=N(s,\L'),
	\]
	for any integer $s$ and, in particular, for $s=r,r+1$.  Applying Theorem~\ref{thm:legalcount}, the result follows.  
\end{proof}

\subsection{Examples} 
\label{sub:examples}

We now give a few examples that illustrate the utility of Theorem~\ref{thm:legalwt}.  Recall that $\ds(G)$ is the number of dominating sets in $G$.  We also write $\ds_k(G)$ for the number of dominating sets in $G$ of size $k$.  Similarly, we define $\sds(G)$ and $\sds_k(G)$ to be, respectively, the number of strong dominating sets in $G$ and the number of strong dominating sets in $G$ of size $k$.  We define the \emph{dominating set polynomial}, $D_G(\mu)$, and the \emph{strong dominating set polynomial}, $D_G^s(\mu)$, of a graph $G$ to be the generating functions enumerated by size, i.e.,
\begin{align*}
	D_G(\mu)&=\sum_k \ds_k(G)\mu^{k},\\
	D_G^s(\mu)&=\sum_k \sds_k(G)\mu^{k}.
\end{align*}
We have the following Corollary of Theorem~\ref{thm:legalwt}.

\begin{corollary}
	If $G$ is an $r$-regular graph on $n$ vertices, and $\mu>0$, then
	\begin{align*}
		D_G(\mu)&\leq D_{K_{r+1}}(\mu)^{n/(r+1)},\\
		D_G^s(\mu)&\leq D_{K_{r,r}}^s(\mu)^{n/(2r)}.
	\end{align*}
\end{corollary}

\begin{proof}
	Let $K=\set{0,1}$ and $\la:K\to (0,\infty)$ be defined by $\la(0)=1$ and $\la(1)=\mu$.  Also, let $\D$ be the collection of all multisets from $K$ containing at least one $1$.  Then $D_G(\mu)=\llc(G,\D)$ and $D_G^s(\mu)=\ll(G,\D)$.  Thus, by Theorem~\ref{thm:legalwt}, we are done.
\end{proof}

A broad class of other applications come from vertex colorings of hypergraphs associated to a graph $G$.  Given a graph $G$, we define $\N(G)$ to be the hypergraph on vertex set $V(G)$ with edge set $\setof{N(v)}{v\in V(G)}$.  Note that if $G$ is $r$-regular, then $\N(G)$ is $r$-uniform.  Analogously, we define $\Nc(G)$ to be the hypergraph of the closed neighborhoods of $G$, so that if $G$ is $r$-regular, $\Nc(G)$ is $(r+1)$-uniform.  Vertex colorings of $G$ are also vertex colorings of $\N(G)$ and $\Nc(G)$.  Various conditions on these hypergraph colorings give rise to invariants of $G$ that are amenable to our techniques.

Recall that a vertex coloring of a hypergraph $\Hy$ is \emph{proper} if no edge is monochromatic, and \emph{rainbow} if no edge contains a repeated color.  We let $\chi(\Hy;q)$ be the number of proper $q$-colorings of $\Hy$ and $\chr(\Hy;q)$ be the number of rainbow colorings of $\Hy$.

\begin{theorem}\label{thm:prorain}
	Let $G$ be an $r$-regular graph on $n$ vertices and $q\geq 1$ be an integer, then
	\begin{align*}
		\chi(\N(G);q)&\leq \chi(\N(K_{r,r});q)^{n/(2r)}=(q^r-q)^{n/r},\\
		\chi(\Nc(G);q)&\leq \chi(\Nc(K_{r+1});q)^{n/(r+1)}=(q^{r+1}-q)^{n/(r+1)},\\
		\chr(\N(G);q)&\leq \chr(\N(K_{r,r});q)^{n/(2r)}=(q(q-1)\cdots(q-r+1))^{n/r},\\
		\chr(\Nc(G);q)&\leq \chr(\Nc(K_{r+1});q)^{n/(r+1)}=(q(q-1)\cdots(q-r))^{n/(r+1)}.
	\end{align*}
\end{theorem}

\begin{proof}
	Let $K=[q]$, $\Pr$ be the collection of all multisets from $K$ that contain at least two different colors, and $\Ra$ be the collection of all subsets of $K$, i.e., multisets all of whose elements are distinct.  The number of proper $q$-colorings of $\N(G)$ (respectively, $\Nc(G)$) is exactly $\l(G,\Pr)$ (respectively, $\lc(G,\Pr)$).  Similarly, the number of rainbow $q$-colorings of $\N(G)$ (resp., $\Nc(G)$) is exactly $\l(G,\Ra)$ (resp., $\lc(G,\Ra)$).  Thus, the result follows from Theorem~\ref{thm:legalcount}.
\end{proof}

Note that a weighted version of Theorem~\ref{thm:prorain} follows equally immediately from Theorem~\ref{thm:legalwt}.



\section{Existence homomorphisms} 
\label{sec:xhoms}

Some of our earlier results are rather naturally phrased in terms of what we call ``existence homomorphisms''.  

\begin{definition}
	Suppose that $G$ and $H$ are graphs with $H$ possibly having loops.  We say that a map $\phi:V(G)\to V(H)$ is an \emph{existence homomorphism} if, for every $v\in V(G)$, there is a $w\in N(v)$ such that $\phi(v)\phi(w)\in E(H)$.  We let $\XHom(G,H)$ be the set of all existence homomorphisms from $G$ to $H$ and set $\xhom(G,H)=\abs{\XHom(G,H)}$.
\end{definition}

\begin{example}
	Let $H$ be $\hind$ with $0$ the unlooped vertex and $1$ the looped one.  If $G$ is a graph of minimum degree at least one, then elements of $\XHom(G,H)$ are precisely the characteristic functions of dominating sets of $G$.  Every vertex that is mapped to $0$ must have some neighbor that is mapped to $1$.  Since vertices mapped to $1$ have some neighbor in the graph, the existence homomorphism condition is satisfied at those vertices\footnote{Note that if $G$ does contain isolated vertices, then $\XHom(G,\hind)=0$.  However, if $S$ is the set of isolates in $G$, then $\ds(G)=\ds(G\wo S)=\xhom(G\wo S)$.}.  Thus, $\xhom(G,H)=\ds(G)$.
\end{example}

\begin{example}
	Existence homomorphisms into complete graphs correspond to proper closed neighborhood colorings.  To be precise, if $\phi\in \XHom(G,K_q)$, then for each vertex $v$, there must be a neighbor $w$ with $\phi(v)\neq \phi(w)$.  Hence, $\phi$ is a proper closed neighborhood coloring of $G$, i.e., a proper vertex coloring of $\Nc(G)$.  Conversely, if $\phi$ is a proper closed neighborhood coloring of $G$, then for all vertices $v$, at least two colors appear on $N[v]$.  In particular, it cannot be that all the colors on $N(v)$ are the same as $\phi(v)$.  Thus, $\phi\in \XHom(G,K_q)$.  Summarizing, we have $\xhom(G,K_q)=\chi(\Nc(G);q)$. 
\end{example}

There seem to be some very interesting problems involving maximizing $\xhom(G,H)$ for fixed $H$.  We discuss below two examples of such problems.  In both cases, we consider $H=\Eq$, the graph on $q$ vertices with a loop on every vertex and no other edges.  We can think of an existence homomorphism into $\Eq$ as a $q$-coloring in which each vertex is adjacent to another vertex of the same color.  In other words, an ordered partition of $G$ into $q$ parts in which every non-trivial part has minimum degree at least one.

The first problem we discuss is that maximizing $\xhom(T,\Eq[2])$ where $T$ is a tree on $n$ vertices.  For convenience, we let $\id(G)=\xhom(G,\Eq[2])$.  We need, first, a simple lemma about $\id(P_n)$.  Clearly, $\id(G)$ is even for every graph $G$ since switching the colors in any existence homomorphism into $\Eq[2]$ gives another such.  In this section, it will be convenient for us to set $i_n=\frac{1}2 \id(P_n)$.  We will think of $i_n$ as being the number of $\phi\in \XHom(P_n,\Eq[2])$ having a specified color at the left-hand end.  We denote the Fibonacci numbers as $F_n$, where $F_0=1$, $F_1=1$, and $F_n=F_{n-1}+F_{n-2}$ for $n\geq 2$.

\begin{lemma}
	For any integer $n\geq 2$, 
	\[
		i_n=F_{n-2}.
	\]
\end{lemma}

\begin{proof}
	It is simple to check that $i_2=i_3=1$.  We label the vertices of $\Eq[2]$ as $a$ and $b$.  If $n\geq 4$ and $\phi\in \XHom(G,\Eq[2])$ uses $a$ on the left-hand end, then the adjacent vertices can be colored either $a$ then $a$, or $a$ then $b$.  There are exactly $i_{n-1}$ of the first type and $i_{n-2}$ of the second.
\end{proof}

\begin{proposition}
	If $T$ is a tree on $n$ vertices, then
	\[
		\id(T)\leq \id(P_n),
	\]
	with equality if and only if $T=P_n$.
\end{proposition}

\begin{proof}
	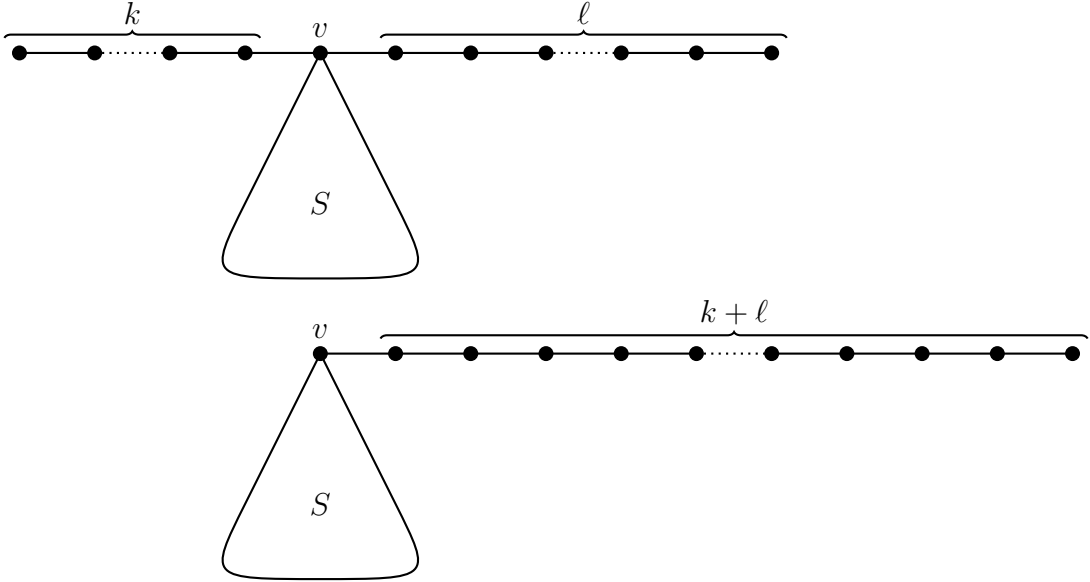
\begin{figure}[ht]
		\begin{tikzpicture}[
				auto=left,
				every path/.style={thick},
				vx/.style={circle,fill=black,inner sep=2pt},
				brace label/.style={midway,inner sep=5pt}
			]
			\draw (0,0) -- (1,-2) ..controls (1.5,-3)  .. (0,-3)
								  		 ..controls (-1.5,-3) .. (-1,-2)
										 -- cycle;
			\node at (0,0.3) {$v$};
			\node at (0,-2) {$S$};
			\draw (-4,0) -- (-3,0) (-2,0) -- (3,0) (4,0) -- (6,0);
			\draw[dotted] (-3,0)--(-2,0) (3,0)--(4,0);
			\draw[decoration={brace}, decorate] (-4.2,0.2) -- (-0.8,0.2) node[brace label] {$k$};
			\draw[decoration={brace}, decorate] (0.8,0.2) -- (6.2,0.2) node[brace label] {$\ell$};
			\foreach \i in {-4,-3,...,6}
				\node[vx] at (\i,0) {};
			\begin{scope}[yshift=-4cm]
				\draw (0,0) -- (1,-2) ..controls (1.5,-3)  .. (0,-3)
									  		 ..controls (-1.5,-3) .. (-1,-2)
											 -- cycle;
				\node at (0,0.3) {$v$};
				\node at (0,-2) {$S$};
				\draw (0,0) -- (5,0) (6,0) -- (10,0);
				\draw[dotted] (5,0)--(6,0);
				\draw[decoration={brace}, decorate] (0.8,0.2) -- (10.2,0.2) node[brace label] {$k+\ell$};
				\foreach \i in {0,1,...,10}
					\node[vx] at (\i,0) {};
			\end{scope}
		\end{tikzpicture}
		\caption{The basic operation transforming $T$ into $T'$}\label{fig:basicop}
	\end{figure}
	We will prove that if $T$ is a tree on $n$ vertices maximizing $\id$, then $T$ cannot contain a vertex of degree at least three having at least two pendant paths.  This, then, by a simple argument, proves that $T=P_n$.  Suppose that, as in Figure~\ref{fig:basicop}, $v$ is a vertex of degree at least three having pendant paths of lengths $k$ and $\ell$.  We let $T'$ be the tree in which these paths are replaced by a single pendant path incident with $v$ of length $k+\ell$.  We will prove that $\id(T')>\id(T)$.  Let us write $S$ for $T$ with the pendant paths, but not $v$, removed (so that $v\in S$).  Since $d_T(v)\geq 3$, we have $n(S)>1$.  For convenience, we will count legal colorings of $T$ and $T'$ in which $v$ is colored $0$.  Let us write $i_S=\frac{1}2 \id(S)$ for the number of legal colorings of $S$ in which $v$ is colored $0$.  Also, we set $j_S$ to be the number of $\set{0,1}$-colorings of $S$ in which every vertex \emph{other than $v$} has a neighbor of the same color and $v$ is colored $0$.  We analyze legal colorings of $T$ and $T'$ according to whether the paths contain a vertex colored $0$ adjacent to $v$.  Thus,
	\begin{align*}
		\frac{1}2\id(T)&=j_S(i_{k+1}i_{\ell+1}+i_{k}i_{\ell+1}+i_{k+1}i_{\ell})+i_S i_k i_{\ell},\\
		\frac{1}2\id(T')&=j_S i_{k+\ell+1}+i_S i_{k+\ell}.
	\end{align*}
	These follow from the fact that the number of legal colorings of a pendant path of length $m$ in which the vertex adjacent to $v$ is colored $0$ is $i_{m+1}$, whereas the number in which that vertex is colored $1$ is $i_m$.  Now, using standard facts about the Fibonacci numbers, we have
	\begin{align*}
		i_{k+\ell+1}&=F_{k+\ell-1}=F_k F_{\ell-1}+F_{k-1} F_{\ell-2}=F_{k-1} F_{\ell-1} + F_{k-2} F_{\ell-1}+F_{k-1} F_{\ell-2}\\
		&=i_{k+1}i_{\ell+1}+i_k i_{\ell+1}+i_{k+1}i_{\ell}, \shortintertext{and}
		i_{k+\ell}&=F_{k+\ell-2}=F_{k-1}F_{\ell-1}+F_{k-2}F_{\ell-2}=i_{k+1}i_{\ell+1}+i_k i_{\ell}\\
		&> i_k i_{\ell}.
	\end{align*}
	Thus, $\id(T')>\id(T)$ provided $\id(S)>0$, i.e., $n(S)>1$.
\end{proof}

Finally, we mention a simple result giving a bound for $\xhom(G,\Eq[2])$ for $2$-regular graphs.  It would be very interesting to give a corresponding bound for $r$-regular graphs.

\begin{proposition}
	If $G$ is a $2$-regular graph on $n$ vertices, then
	\[
		\xhom(G,\Eq[2])\leq \xhom(C_6,\Eq[2])^{n/6}.
	\]
\end{proposition}

\begin{proof}
	In \cite{AFK,AFKerr}, Agur, Fraenkel, and Klein prove that $c_n=\xhom(C_n,\Eq[2])$ satisfies the recurrence 
	\[
		c_n=2c_{n-1}-c_{n-2}+c_{n-4},\qquad n\geq 7.
	\]
	From this, it is straightforward to determine that $c_n^{1/n}\to \phi$, the golden ratio, as $n\to \infty$.  Checking small values, one finds that $c_n^{1/n}$ is maximized at $n=6$.  Since $\xhom(\cdot,\Eq[2])$ is multiplicative on disjoint unions, this proves the result.
\end{proof}


\bibliographystyle{amsplain}
\bibliography{existhom}

\end{document}